\documentclass{amsart}

\newcommand{\Cone}{\mathrm{Cone}}
\newcommand{\con}{\mathrm{conv}}

\newcommand{\IR}{\mathbb R}
\newcommand{\IC}{\mathbb C}
\newcommand{\IT}{\mathbb T}
\newcommand{\Ra}{\Rightarrow}
\newcommand{\IN}{\mathbb N}
\newcommand{\w}{\omega}
\newcommand{\I}{\mathcal I}
\newcommand{\U}{\mathcal U}
\newcommand{\V}{\mathcal V}
\newcommand{\e}{\varepsilon}

\newtheorem{theorem}{Theorem}
\newtheorem{example}{Example}
\newtheorem{lemma}{Lemma}
\newtheorem{proposition}{Proposition}
\newtheorem{corollary}{Corollary}
\newtheorem{question}{Question}

\begin{document}

\title[compactly convex sets in linear topological spaces]{compactly convex sets in linear topological spaces}         
\author{T.~Banakh, M.~Mitrofanov, O.~Ravsky}        
\date{}          

\begin{abstract} A convex subset $X$ of a linear topological space is called {\em compactly convex} if there is a continuous compact-valued map $\Phi:X\to\exp(X)$ such that $[x,y]\subset\Phi(x)\cup\Phi(y)$ for all $x,y\in X$. We prove that each convex subset of the plane is compactly convex. On the other hand, the space $\IR^3$ contains a convex set that is not compactly convex. Each compactly convex subset $X$ of a linear topological space $L$ has locally compact closure $\bar X$ which is metrizable if and only if each compact subset of $X$ is metrizable.
\end{abstract}
\address{T.Banakh: Ivan Franko National University of Lviv (Ukraine) and Jan Kochanowski University, Kielce (Poland)}
\address{M.Mitrofanov, O.Ravsky: Institute of Applied Problems of Mathematics and Mechanics of Ukrainian Academy of Science, Lviv, Ukraine}

\maketitle

This paper is devoted to so-called  compactly convex sets in linear topological spaces. A convex subset $X$ of a linear topological space $L$ is defined to be {\em compactly convex} if there is a continuous compact-valued map $\Phi$ assigning to each point $x\in X$ a compact subset $\Phi(x)\subset X$ so that $[x,y]\subset\Phi(x)\cup\Phi(y)$ for any points $x,y\in X$.

The continuity of the multi-valued map $\Phi$ is understood in the sense of the Vietoris topology on the hyperspace $\exp(X)$ of all non-empty compact subsets of $X$. The sub-base of this topology consists of the sets $$\langle U\rangle=\{K\in\exp(X):K\subset U\}\quad\mbox{and}\quad\rangle U\langle=\{K\in\exp(X):K\cap U\ne\emptyset\},$$ where $U$ runs over open subsets of $X$.
If the space $X$ is metrizable by a metric $\rho$, then the Vietoris topology on $\exp(X)$ is metrizable by the Hausdorff metric $$\rho_H(A,B)=\max\{\max_{a\in A}\rho(a,B),\max_{b\in B}\rho(b,A)\}.$$ A map $\Phi:X\to\exp(X)$ is called upper semi-continuous if it is continuous with respect to the upper topology on $\exp(X)$ generated by the base consisting of the sets $\langle U\rangle$ where $U$ runs over open subsets of $X$.

The interest in compactly convex sets came from Theory of Retracts, where the following problem still stands open: is the direct limit of stratifiable absolute retracts an absolute retract for stratifiable spaces? For direct limits of metrizable absolute retracts the positive answer to this problem was given in the paper \cite{Ban}. The proof relied on the notion of a {\em compactly equi-connected} space, that is a topological space $X$ endowed with an equiconnecting function $\lambda:X\times X\times[0,1]\to X$ such that for some upper semi-continuous compact-valued map $\Phi:X\Ra X$ the inclusion $\lambda(x,y,t)\subset\Phi(x)\cup\Phi(y)$ holds for all $(x,y,t)\in X\times X\times[0,1]$. By an equiconnected function on $X$ we understand a continuous map $\lambda:X\times X\times[0,1]\to X$ such that $\lambda(x,y,0)=x$, $\lambda(x,y,1)=y$ and $\lambda(x,x,t)=x$ for all $(x,y,t)\in X\times X\times [0,1]$.

In light of this notion it became interesting to detect compactly equiconnected spaces among convex sets endowed with the natural equi-connected structure. This problem turned to be not trivial even in case of convex subsets of the plane.

The principal our result is

\begin{theorem}\label{main} Any convex subset of the plane is compactly convex.
\end{theorem}

This result is specific for low dimensional convex sets and in the dimension three we have a counterexample.

\begin{example}\label{ex1} There is a convex subset $X\subset\IR^3$ which is not compactly convex.
\end{example}

\begin{proof} Identify $\IR^3$ with the product $\IC\times\IR$ of the complex plane and the real line. Let $D=\{z\in\IC:|z|<1\}$ be the open unit disc and $\IT=\{z\in\IC:|z|=1\}$ be its boundary. Take any three pairwise disjoint dense subsets  $T_0,T_I,T_1$ in $\IT$. Let $I=[0,1]$ stand for the unit interval.

We claim that the convex subset
$$
X=(D\times I)\cup(T_I\times I)\cup(T_0\times\{0\})\cup(T_1\times\{1\})$$ of $\IC\times\IR$ fails to be compactly convex.

Assuming that it is compactly convex, find a (upper semi-) continuous compact-valued map  $\Phi:X\to\exp(X)$ such that $[x,y]\subset \Phi(x)\cup\Phi(y)$.  For every $x\in \IT$ and $t\in[0,1]$ let $x_t=(x,t)$. If $x\in T_I$, then $x_0,x_1\in X$ and hence $[x_0,x_1]\subset\Phi(x_0)\cup\Phi(x_1)$. Then $x_{1/2}\in\Phi(x_0)\cup \Phi(x_1)$ and thus $T_I=T^0_I\cup T^1_I$ where$$T_I^i=\{x\in T_I:x_{1/2}\in \Phi(x_i)\}$$for $i\in\{0,1\}$. Since $T_I^0\cup T_I^1$ is dense in $\IT$ there are $i\in\{0,1\}$ and a non-empty open set $U\subset\IT$ such that $U\cap T_I^i$ is dense in $U$. Take any point $y\in U\cap T_i$. Since $x_{1/2}\in\Phi(x)$ for every $x\in U\cap T^i_I$, the upper semicontinuity of $\Phi$ implies that $y_{1/2}\in \Phi(y)\subset X$, which is a contradiction because $y_{1/2}\notin X$.
\end{proof}

Theorem~\ref{main} will be proven in several steps. Firstly for bounded convex sets, then for locally compact, and finally with help of Sum Theorems, for all convex subsets of the plane. For constructing continuous compact-valued maps with given properties we shall often use

\begin{lemma}\label{l1}  Let $X$ be a normal topological space, $C$ be a contractible Tychonoff space, and $\{U_i\}_{i\in \I}$ be a locally finite open cover of $X$. Let $\{F_i\}_{i\in\I}$ be a family of closed subsets of $X$ such that $F_i\subset U_i$ for all $i\in\I$. Then for any continuous maps $\Phi_i:U_i\to\exp(C)$, $i\in \I$, there is a continuous map $\Phi:X\to \exp(C)$ such that $\Phi_i(x)\subset\Phi(x)$ for all $x\in F_i$, $i\in\I$.
\end{lemma}

\begin{proof}  Using the normality of $X$, for every $i\in\I$ fix a continuous function $\lambda_i:X\to [0,1]$ such that $\lambda_i(F_i)\subset\{1\}$ and $\lambda_i(X\setminus U_i)\subset\{0\}$.

Since $C$ is contractible, there is a point $c_0\in C$ and a homotopy $h:C\times [0,1]\to C$  such that $h(c,0)=c_0$ and $h(c,1)=c$ for all $c\in C$. The homotopy $h$ induces a homotopy $H:\exp(C)\times[0,1]\to\exp(C)$ acting by the formula: $H(K,t)=\{h(x,t):x\in K\}$ for $(K,t)\in \exp(C)\times[0,1]$.

Finally, define the map $\Phi:X\to\exp(C)$ letting $$\Phi(x)=\bigcup_{i\in I}H(\Phi_i(x),\lambda_i(x))\mbox{ \ for \ $x\in X$.}$$ It is easy to check that this map is well-defined, continuous and has the desired property: $\Phi_i(x)\subset \Phi(x)$ for all $x\in F_i$, $i\in\I$.
\end{proof}

\section{The compact convexity of bounded convex sets}

The main result of this section is

\begin{lemma}\label{p1} Each bounded convex subset of the plane is compactly convex.
\end{lemma}

\begin{proof} Let $C$ be a bounded convex subset of the plane, $\bar C$ be its  closure and $\partial C$ be its boundary in $\IR^2$. If the boundary $\partial C$ contains no non-degenerate interval, then the solution is simple: just consider the compact-valued map $\Phi:x\mapsto \frac12x+\frac12\bar C$ and observe that it turns $C$ into a compactly convex set.

If $\partial C$ contains non-degenerate intervals, then a more subtle construction will be supplied. In the square $\bar C^2$ consider the set $$D=\{(x,y)\in\bar C^2:[x,y]\subset\partial C,\; x\ne y,\; x\in C\}$$ and let $\bar D$ be the closure of $D$ in $\bar C^2\setminus\Delta$, where $\Delta=\{(x,y)\in\bar C^2:x=y\}$ is the diagonal of $\bar C^2$. Let $\bar D^{-1}=\{(x,y):(y,x)\in\bar D\}$.

Next, consider the (continuous) function $\lambda:\bar D\to[0,1]$ defined by $\lambda(x,y)=\frac{|[x,y]\cap C|}{|[x,y]|}$, where $|[x,y]|$ is the length of the interval $[x,y]$. This function yields the relative length of the intersection $[x,y]\cap C$ for a pair $(x,y)\in D$. Note that $\lambda(x,y)=1$ for any $(x,y)\in\bar D\cap\bar D^{-1}$.

It is well-known that for any continuous  function $\alpha:B\to[0,1]$ defined on a closed subset of a metrizable space $M$ admits a continuous extension $\bar \alpha:M\to[0,1]$ which is strictly positive on the complement $M\setminus B$.

Using this extension property of the metrizable space $\bar C^2\setminus\Delta$, find a continuous function $f:\bar C^2\setminus \Delta$ such that
\begin{itemize}
\item $f(x,y)=\frac12\lambda(x,y)$ for any $(x,y)\in\bar D$;
\item $f(x,y)=1-\frac12\lambda(y,x)$ for any $(x,y)\in\bar D^{-1}$;
\item $f(x,y)>0$ for any $(x,y)\in\bar C^2\setminus(\Delta\cup\bar D\cup\bar D^{-1})$.
\end{itemize}

Moreover, replacing $f(x,y)$ by $\frac{f(x,y)}{f(x,y)+f(y,x)}$, if necessary, we can additionally assume that $f(x,y)+f(y,x)=1$ for all $(x,y)\in\bar C^2\setminus\Delta$.

It follows from the choice of $f$ that $f(x,y)=0$ iff $(x,y)\in\bar D$ and $\lambda(x,y)=0$.

Now define a compact-valued map $\Phi:C\to\exp(\bar C)$ by
$$\Phi(x)=\{x\}\cup\{ty+(1-t)x:y\in \bar C\setminus\{x\},\; 0\le t\le f(x,y)\}.$$

It is easy to check that $\Phi:C\to\exp(\bar C)$ is a well-defined compact-valued map, continuous with respect to the Vietoris topology on the hyperspace of $\bar C$. The inclusion $[x,y]\subset\Phi(x)\cup\Phi(y)$ follows from the equality $f(x,y)+f(y,x)=1$ holding for all $x,y\in\bar C^2$, $x\ne y$.

Finally, let us show that $\Phi(x)\subset C$ for all $x\in C$. Take any $y\in \bar C\setminus\{x\}$ and a positive $t\le f(x,y)$. We should show that $ty+(1-t)x\in C$. First, we verify that $t<1$. Assuming that $t=1$, we get $f(x,y)=1$ and $f(y,x)=0$ which implies $(y,x)\in\bar D$ and $\lambda(y,x)=0$. Let $(a,b)\in D$ be a pair so close to the pair $(y,x)\in\bar D$ that the intersection $[x,y]$ and $[a,b]$ contains a non-degenerate interval. Since $a,x\in C$, the length $[y,x]\cap C\supset [y,x]\cap [a,x]$ is strictly positive and thus $\lambda(y,x)\ne0$, which is a contradiction.

Therefore, $t<1$. If $[x,y]\not\subset\partial C$, then $ty+(1-t)x$ lies in the interior of $C$.
Finally, consider the case $[x,y]\subset\partial C$. Then $t\le f(x,y)=\frac12\lambda(x,y)$ and thus $ty+(1-t)x\in [x,y]\cap C$.
\end{proof}

\section{The compact convexity of locally compact convex sets}

\begin{proposition}\label{p2} Each locally compact $\sigma$-compact convex subset $C$ of a linear topological space $L$ is compactly convex.
\end{proposition}

\begin{proof}
Since $C$ is locally compact and $\sigma$-compact, we can write $C$ as the sum $C=\bigcup_{n\in\w} C_n$ of an increasing sequence of compacta such that each $C_n$ lies in the interior $U_{n+1}$ of $C_{n+1}$ in $C$. It will be convenient to put $C_n=\emptyset$ for $n<0$. Since $(U_{n+1}\setminus C_{n-2})_{n\in\w}$ is a locally finite open cover of $C$, we may apply Lemma~\ref{l1} to construct a continuous compact-valued map $\Phi:C\to\exp(C)$ such that for every $n\in\w$ and $x\in C_n\setminus C_{n-1}$ the set $\Phi(x)$ contains the set $[C_n,C_n]=\bigcup_{a,b\in C_n}[a,b]$.

 We claim that $\Phi$ witnesses the compact convexity of $C$. Indeed, given any points $x,y\in C$, find a minimal $n$ such that $\{x,y\}\subset C_n$. Then $x$ or $y$ belongs to $C_n\setminus C_{n-1}$. Without loss of generality, $x\in C_n\setminus C_{n-1}$. Then $\Phi(x)\supset [C_n,C_n]$ and hence $$[x,y]\subset[C_n,C_n]\subset\Phi(x)\subset\Phi(x)\cup\Phi(y).$$
\end{proof}

We do not know if the $\sigma$-compactness  is essential in the preceding proposition.

\begin{question} Is any locally compact convex subset of a linear topological space compactly convex?
\end{question}

\section{A Subspace Theorem for compactly convex sets}

\begin{theorem} A closed convex subset $C$ of a compactly convex metrizable set $X$ is compactly convex.
\end{theorem}

\begin{proof} Let $\Phi:X\to\exp(X)$ be a continuous compact-valued map witnessing the compact convexity of $X$. By the Dugunji Theorem \cite{Dug}, the closed convex set $C$ is an absolute retract. Consequently, there is a retraction $r:X\to C$ inducing a retraction $\exp(r):\exp(X)\to\exp(C)$. Then the compact-valued map $\exp(r)\circ\Phi|C:C\to\exp(C)$ witnesses the compact convexity of $C$.
\end{proof}

We do not know if the metrizability is essential in the preceding proposition.

\begin{question} Is any closed convex subset of a compactly convex set compactly convex?
\end{question}

\section{Sum Theorems for compactly convex sets}

\begin{theorem}\label{t2} Let $P_-,P_+$ be two closed half-spaces in a linear topological space $L$, intersecting by a hyperplane $l=P_-\cap P_+$. A convex subset $C\subset L$ is compactly convex if  the sets $C\cap P_-$ and $C\cap P_+$ are compactly convex and there is a metrizable locally compact convex subset $\tilde C\supset C$ such that $\tilde C\cap l=C\cap l$.
\end{theorem}

\begin{proof} There is noting to prove if $C\subset C\cap P_-$ or $C\subset C\cap P_+$.
So we assume that both the sets $C\setminus P_-$ and $C\setminus P_+$ are not empty and so is the intersection $C\cap l$.

Being convex closed subsets of the metrizable space $C$, the sets $C\cap P_-$ and $C\cap P_+$ are retracts in $C$, see \cite{Dug}. This fact and the compact convexity of the sets $C\cap P_-$ and $C\cap P_+$ allow us to construct a continuous map $\Psi:C\to\exp(C)$ such that
$[x,y]\subset \Psi(x)\cup\Psi(y)$ for every points $x,y\in C$ that both belong either to $C\cap P_-$ or to $C\cap P_+$.

Being locally compact, the convex set $\tilde C$ can be written as the union $\tilde C=\bigcup_{n\in\w}C_n$ of a sequence  of compact subsets $C_n$, $n\in\w$, such that $C_0\cap l\ne\emptyset$ and each set $C_n$ lies in the interior of $C_{n+1}$ in $\tilde C$. It follows that for every $n\in\w$ the set $[C_n,C_n]=\{(1-t)a+tb:t\in[0,1],\;a,b\in C_n\}$ is compact and so is the intersection $[C_n,C_n]\cap l\subset\tilde C\cap l=C\cap l$.

Using Lemma~\ref{l1} we can construct a continuous map $\Phi:C\to\exp(C)$ such that for every $x\in C$
\begin{itemize}
\item $\Phi(x)\supset \Psi(x)$;
\item $\Phi(x)\supset \Psi([C_n,C_n]\cap l)$ provided $x\in C_n\setminus C_{n-1}$, $n\in\w$.
\end{itemize}

We claim that $[x_0,x_1]\subset\Phi(x_0)\cup\Phi(x_1)$ for any $x_0,x_1\in C$. This is clear if  both the points $x_0,x_1$ lie in $C\cap P_-$ or $C\cap P_+$. So we assume that $x_0\in C\cap P_-$ and $x_1\in C\cap P_+$. It follows that the intersection $[x_0,x_1]\cap l$ contains some point $z$. The choice of the function $\Psi$  guarantees that $[x_0,z]\subset\Psi(x_0)\cup\Psi(z)$ and $[z,x_1]\subset \Psi(z)\cup\Psi(x_1)$.

Find a minimal $n$ such that $\{x_0,x_1\}\subset C_n$. Then for some $i\in\{0,1\}$ the point $x_i$ belongs to $C_n\setminus C_{n-1}$ and hence $\Phi(x_i)\supset \Psi([C_n,C_n]\cap l)\supset\Psi(z)$. Now we see that
$$
\begin{aligned}
{[x_0,x_1]}=[x_0,z]\cup[z,x_1]\subset &\; \Psi(x_0)\cup\Psi(z)\cup\Psi(z)\cup\Psi(x_1)\subset\\
\subset&\; \Phi(x_0)\cup\Phi(x_i)\cup\Phi(x_1)=\Phi(x_0)\cup \Phi(x_1).
\end{aligned}$$
\end{proof}

\begin{corollary}\label{cc} A convex subset $C\subset \IR^2$ is compactly convex provided $C$ is the union of two closed convex compactly convex subsets whose intersection lies on a line.
\end{corollary}

\begin{proof} Let $C_-,C_+$ be closed compactly convex subsets of $C$ whose union coincides with $C$ and the intersection $C_-\cap C_+$ lies on some line $l$.
The compact convexity of the set $C$ will follow from Theorem~\ref{t2} as soon as we find a locally compact set $\tilde C\subset\IR^2$ such that $C\subset\tilde C$ and $C\cap l=\tilde C\cap l$.

Let $B$ be the boundary of the set $C\cap l$ in $l$. It is clear that $|B|\le 2$. For every point $b\in B$ use the geometric form of the Hahn-Banach Theorem to find a closed half-plane $P_b$ such that $C\subset P_b$ and $b$ belongs to the boundary line $L_b$ of $P_b$. Write the line $L_b$ as the disjoint union $L_b=L^0_b\cup L_b^1$ of two closed half-lines with $L_b^0\cap L^1_b=\{b\}$. If $b\in C$, then put $\tilde P_b=P_b$.
If $b\notin C$, then $L_b^i\cap C=\emptyset$ for some $i\in\{0,1\}$ (by the convexity of $C$). In this case let $\tilde P_b=P_b\setminus L_b^i$. In both the cases the set $\tilde P_b$ is locally compact. The same is true for the intersection $\tilde C=\bigcap_{b\in B}\tilde P_b$. It is clear that $C\subset\tilde C$ and $\tilde C\cap l=C\cap l$.
\end{proof}

In the sequel we shall also need the Infinite Sum Theorem:

\begin{theorem}\label{t3} A convex normal subspace $C$ of a linear topological space $L$  is compactly convex if $C$ is the union of an increasing sequence $(C_n)_{n\in\w}$ of closed compactly convex subsets of $C$ such that  each set $C_n$ lies in the interior of $C_{n+1}$ and has compact boundary $\partial C_n$ in $C$.
\end{theorem}

\begin{proof} For every $n\in\w$ the compact convexity of $C_n$ yields us a  continuous compact-valued map $\Phi_n:C_{n}\to\exp(C)$ such that $[x,y]\subset \Phi_n(x)\cup\Phi_n(y)$ for all $x,y\in C_{n}$. For $n<0$ we put $C_n=\emptyset$.

Using Lemma~\ref{l1}, construct a continuous compact-valued function $\Phi:C\to\exp(C)$ such that for every $n\in\w$ and $x\in C_n\setminus C_{n-1}$ the following conditions are satisfied:
\begin{enumerate}
\item $\Phi(x)\supset\Phi_{n+1}(x)$;
\item $\Phi(x)\supset\Phi_{n+1}(\partial C_n\cup\partial C_{n-1})$;
\item $\Phi(x)\supset[\partial C_{k-1},\partial C_{k}]$ for all $k<n$.
\end{enumerate}

We claim that $\Phi$ witnesses the compact convexity of $X$. Take any two points $x,y\in C$ and find numbers $k,n\in\w$ with $x\in C_k\setminus C_{k-1}$ and $y\in C_n\setminus C_{n-1}$. Without loss of generality, $k\le n$. We should prove that $[x,y]\subset \Phi(x)\cup\Phi(y)$. Observe that for every $i\in\w$ with $k\le i<n$ the intersection of the interval $[x,y]$ with the boundary $\partial C_i$ contains some point $z_i$. Taking into account the three properties of the map $\Phi$ we conclude that
$$\begin{aligned}
{[x,y]}=&\;[x,z_k]\cup\big(\bigcup_{i=k+1}^{n-1}[z_{i-1},z_{i}]\big)\cup[z_{n-1},y]\subset\\
\subset&\;\Phi_{k+1}(x)\cup\Phi_{k+1}(\partial C_k)\cup\big(\bigcup_{i=k+1}^{n-1}[\partial C_{i-1},\partial C_{i}]\big)\cup \Phi_{n+1}(\partial C_{n-1})\cup \Phi_{n+1}(y)\subset\\
\subset&\;\Phi(x)\cup\Phi(y).
\end{aligned}
$$
\end{proof}

\section{Proof of Main Theorem~\ref{main}}

Let $C$ be a convex subset of the plane and $\overline{C}$ be its closure. If $C$ is locally compact, then $C$ is compactly convex by Proposition~\ref{p2}.

So, we can assume that $C$ is not locally compact and consider the set  $N=\{c\in C:C$ is not locally compact at $c\}$. If the set $N$ is bounded, then we can write $C$ as the union  $C=C_0\cup C_1\cup C_2\cup C_3$ of closed convex subsets of $C$ such that $C_0\supset N$ is bounded, $C_1,C_2,C_3$ are locally compact and $\bigcup_{i=0}^nC_i$ are convex for $n\le 3$:

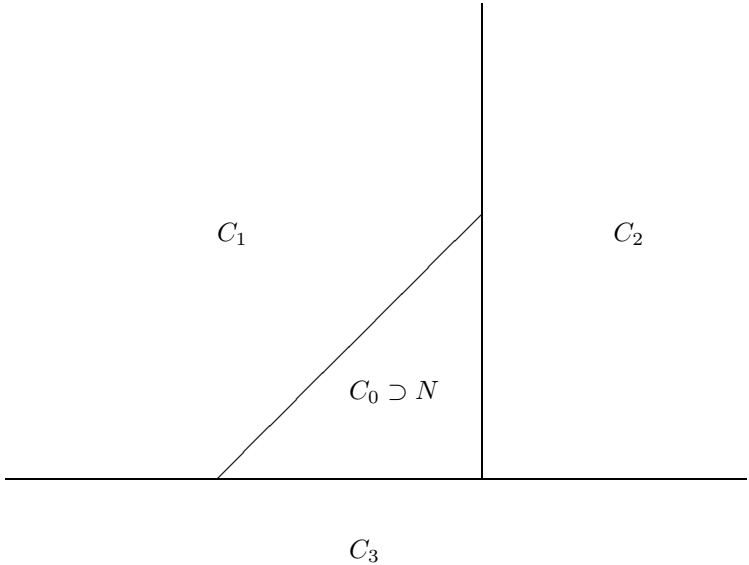
\begin{picture}(300,250)

\put(20,50){\line(1,0){280}}
\put(200,50){\line(0,1){180}}
\put(100,50){\line(1,1){100}}

\put(150,80){$C_0\supset N$}
\put(150,20){$C_3$}
\put(100,140){$C_1$}
\put(250,140){$C_2$}
\end{picture}

By Lemma~\ref{p1}, the set $C_0$ is compactly convex being bounded and by Proposition~\ref{p2}, the convex locally compact sets $C_1,C_2,C_3$ are compactly convex. Applying Corollary~\ref{cc} three times we conclude that the sets $C_0\cup C_1$, $(C_0\cup C_1)\cup C_2$ and finally $C=((C_0\cup C_1)\cup C_2)\cup C_3$ are compactly convex.

Next, we consider the case of unbounded $N$.
In this case $C$ is two dimensional and has  infinite boundary $\partial C$. Shifting $C$, if necessary, we can assume that zero is an interior point of $C$. Consider the convex cone $\Cone(\overline{C})=\{\vec v\in\IR^2: [0,+\infty)\vec v\subset \overline{C}\}$ of $\overline{C}$. We claim that $\Cone(\overline{C})$ is unbounded and contains no line. Indeed, assuming that $\con(\overline{C})$ is bounded we would conclude that the sets $\overline{C}\supset N$ are bounded, which is not the case. Assuming  that $\Cone(\overline{C})$ contains some line, we would get that $\overline{C}$ is either a half-plane or a strip. In both the cases, $N$ would be bounded.

Hence the cone $\Cone(\overline{C})$ is unbounded and contains no line. So by an affine transformation of the plane we can reduce the problem to the case $$[0,\infty)\times\{0\}\subset\Cone(\overline{C})\subset\{(x,y):|y|\le x\}.$$ The real line $\IR\times\{0\}$ divides $C$ into two closed convex subsets $C_-=\{(x,y)\in C:y\le 0\}$ and $C_+=\{(x,y)\in C:y\ge0\}$. According to the Sum Theorem~\ref{t2} it suffices to verify that both the sets $C_-$ and $C_+$ are compactly convex.

If the set $N\cap C_+$ is bounded, then we can prove the compact convexity of $C_+$ decomposing it into  four convex pieces, one of which is bounded and the other locally compact.

If $N\cap C_+$ is not bounded, then we can find a sequence $\{(x_n,y_n)\}_{n=1}^\infty\in N\cap C_+$ tending to infinity and such that $0<x_n<x_{n+1}$ for all $n$. Consider the sets $C_n=\{(x,y)\in C_+:x\le x_n\}$, $n\in\w$, and observe that $(C_n)_{n\in\w}$ is an increasing sequence of bounded closed convex subsets of $C_+$ with compact boundaries in $C$  such that each $C_n$ lies in the interior of $C_{n+1}$ in $C$. By Lemma~\ref{p1}, each set $C_n$, being bounded, is compactly convex. Applying the Infinite Sum Theorem~\ref{t3}, we  conclude that the set $C_+$ is compactly convex as well.

The proof of the compact convexity of the set $C_-$ is analogous.

\section{Compactly convex sets in higher dimensional spaces}

In this section we investigate compactly convex sets in general linear topological spaces. In some cases the requirement of compact convexity can be weakened to the upper compact convexity. We define a convex subset $C$ of a linear topological space $L$ to be {\em upper compactly convex} if there is an upper-semicontinuous compact-valued map $\Phi:C\to\exp(C)$ such that $[x,y]\subset \Phi(x)\cup\Phi(y)$. It is clear that each compactly convex set is upper compactly convex. On the other hand, the space constructed in Example~\ref{ex1} is not upper compactly convex.

\begin{proposition}\label{p2.1} If $X$ is an upper compactly convex subset of a linear topological space $L$, then for any point $x\in X$ there are a neighborhood $V\subset X$ of $x$, a point $z\in X$ and a compact subset $K\subset X$ such that $(V+z)/2\subset K$.
\end{proposition}

\begin{proof} Let $\Phi:X\to\exp(X)$ be an upper-semicontinuous map such that $[y,z]\subset\Phi(y)\cup\Phi(z)$ for all $y,z\in X$. If $\frac{X+x}2\subset \Phi(x)$, then we put $V=X$, $z=x$, and $K=\Phi(x)$.

If $\frac{X+x}2\not\subset \Phi(x)$, then take any point $z\in X$ with $\frac{z+x}2\notin \Phi(x)$ and let $K=\Phi(z)$. Since $\Phi(x)$ is compact, there is a neighborhood $U\subset L$ of the origin such that $\frac{z+x+U}2\cap (\Phi(x)+U)=\emptyset$.

The upper semi-continuity of $\Phi$ yields a neighborhood $V\subset X$ of $x$ such that $\Phi(y)\subset \Phi(x)+U$ for all $y\in V$. Moreover, we can take $V$ so small that $V\subset x+U$. Then for every $y\in V$ the point $\frac{z+y}2\in\frac{z+x+U}2$ does not belong to $\Phi(y)\subset\Phi(x)+U$. On the other hand, the interval $[y,z]$ lies in $\Phi(z)\cup\Phi(y)$ and hence $\frac{z+y}2\in\Phi(z)=K$ which yileds the desired inclusion $\frac{z+V}2\subset K$.
\end{proof}

\begin{corollary}\label{c5} An upper compactly convex subet $X$ of a linear topological space $L$  has metrizable closure iff each compact subset of  $X$ is metrizable.
\end{corollary}

\begin{proof} The ``only if'' part is trivial. To prove the ``if'' part, assume that each compact subset of $X$ is metrizable. We should show that the closure $\overline{X}$ of $X$ in $L$ is metrizable.
After a suitable shift of $X$ we can assume that the origin of $L$ belongs to $X$.
Proposition~\ref{p2.1} yields a neihborhood $V\subset X$ of the origin, a point $z\in X$, and a compact subset $K\subset X$ such that $\frac{V+z}2\subset K$. The set $K$, being a compact subset of $X$, is metrizable. Then so is the set $A=2K-z$ and its subsets $V$ and $\overline{V}$, where $\overline{V}$ stands for the closure of $V$ in $L$. Let $\widetilde  V$ be an open subset of $\overline{X}$ such that $\widetilde  V\cap X=V$. The density of $X$ in $\overline{X}$ implies that $\widetilde  V\subset\overline{V}$.
Since $\widetilde  V$ is an open neighborhood of the origin in $\overline{X}$, for every point $x\in\overline{X}$ there is a positive integer number $n$ with $\frac{x}n\in\widetilde  V$. Now the continuity of multiplication by a real number, yields us a neighborhood $W\subset\overline{X}$ such that $\frac1nW\subset\widetilde  V\subset\overline{V}\subset A$.
Since $A$ is metrizable, so is the neighborhood $W$ of $x$. This shows that the space $\overline{X}$ is locally metrizable. Moreover, we also get $\overline{X}\subset\bigcup_{n\in\IN}nA$. Since $A$ is metrizable and compact, the space $\bigcup_{n\in\IN}nA$ has countable network of the topology and so does its subspace $\overline{X}$. Since spaces with countable network are Lindel\"of and thus paracompact, we conclude that the space $\overline{X}$ is paracompact. Being locally metrizable, this space is metrizable according to \cite[5.4.A]{En}.
\end{proof}

\begin{corollary}\label{c6} The closure of each upper compactly convex set $X$ in a linear topological space $L$ is locally compact.
\end{corollary}

\begin{proof} After a suitable shift of $X$ we can assume that the origin of $L$ belongs to $X$. Repeating the argument of the proof of Corollary~\ref{c5}, we can find
a neighborhood $V\subset \overline{X}$ of the origin having compact closure $K=\overline{V}$ in $L$.
This shows that $\overline{X}$ is locally compact at the origin. The continuity of multiplication by real numbers implies that for every point $x\in\overline{X}$ there is a closed neighborhood $W\subset \overline{X}$ and a number $n\in\IN$ such that $\frac1nW\subset V$ and thus $W\subset nV\subset n\overline{V}=nK$. Being a closed subset of the compact space  $nK$, the neighborhood $W$ is compact, witnessing that $\overline{X}$ is locally compact at $x$.
\end{proof}

\begin{corollary}\label{c4} For a closed convex subset $X$ of a linear topological space the following conditions are equivalent:
\begin{enumerate}
\item $X$ is compactly convex;
\item $X$ is upper compactly convex;
\item $X$ is locally compact.
\end{enumerate}
\end{corollary}

\begin{proof} The implication $(1)\Ra(2)$ is trivial and $(2)\Ra(3)$ follows from Corollary~\ref{c6}. The implication $(3)\Ra(1)$ will follow from Proposition~\ref{p2} and the following known lemma.
\end{proof}

\begin{lemma}\label{l2.1} Each closed convex locally compact subset $X$ of a linear topological space $L$ is $\sigma$-compact.
\end{lemma}

\begin{proof} After a suitable shift we can assume that the origin of $L$ belongs to $X$. The local compactness of $X$ yields a compact neighborhood $K\subset X$ of the origin. For every $x\in X$ the sequence $(\frac{x}n)_{n=1}^\infty$ tends to zero, Consequently, $\frac{x}n\in K$ for some $n\in\IN$. This shows that $X\subset\bigcup_{n=1}^\infty nK$ is $\sigma$-compact, being a closed subspace of the $\sigma$-compact space $\bigcup_{n\in\IN}nK$.
\end{proof}

A convex subset $X$ of a linear topological space $L$ is called {\em strongly convex} if $tx+(1-t)y\in X$ for any $(x,y,t)\in X\times\bar X\times (0,1]$. Let us note that a convex set $X$ is strongly convex if the boundary $\partial X$ of $X$ in $L$ contains no non-degenerate interval.

\begin{theorem}\label{strc} For a strongly convex set $X$ of a linear topological space $L$ the following conditions are equivalent:
\begin{enumerate}
\item $X$ is compactly convex;
\item $X$ is upper compactly convex;
\item the closure of $X$ in $L$ is locally compact.
\end{enumerate}
\end{theorem}

\begin{proof} The implication $(1)\Ra(2)$ is trivial while $(2)\Ra(3)$ follows from Corollary~\ref{c6}. To prove that $(3)\Ra(1)$, assume that $X\subset L$ is a strongly convex set with  locally compact closure $\bar X$ in $L$.
By Lemma~\ref{l2.1}, the set $\bar X$ is $\sigma$-compact and hence can be written as the countable union $\bar X=\bigcup_{n=0}^\infty X_n$ of an increasing sequence of compact subsets of $\bar X$ such that each $X_n$ lies in the interior of $X_{n+1}$.
Let $X_n=\emptyset$ for $n<0$.

For every $k\le m$ use the compactness of the sets $X_k,X_m$ to find a real number $\e_{k,m}\in(0,1]$ such that $$\{(1-t)x+ty:x\in X_k,\;y\in X_m,\; 0\le t\le \e_{k,m}\}\subset X_{k+1}.$$

The square $\bar X^2$ is $\sigma$-compact and thus paracompact. This allows us to find a continuous function $f:\bar X^2\to(0,1)$ such that
\begin{itemize}
\item $f(x,y)\le \frac{\e_{k,m}}m$ for any $k\le m-2$ and points $x\in X_k\setminus X_{k-1}$, $y\in X_{m}\setminus X_{m-1}$;
\item $f(x,y)+f(y,x)=1$ for all $x,y\in\bar X$.
\end{itemize}

For every point $x\in\bar{X}$ consider the set
$$\Phi(x)=\{(1-t)x+ty:y\in\bar X,\; 0\le t\le f(x,y)\}.$$
We claim that this set is compact.
Observe that $\Phi(x)=\bigcup_{n=1}^\infty\Phi_n(x)$ where for every $n\in\IN$ the set $$\Phi_n(x)=\{ty+(1-t)x:y\in X_n,\; 0\le t\le f(x,y)\}$$is compact.

Take any cover $\U$ of $\Phi(x)$ by open subsets of $L$ and find an open set $U_x\in\U$ containing the point $x$. Let $k\in\w$ be such that $x\in X_k\setminus X_{k-1}$.
By the compactness of $X_{k+1}$ there is a number $n>k$ such that $(1-t)x+tX_{k+1}\subset U_x$ for every positive $t\le\frac1n$. We claim that $\Phi(x)\setminus \Phi_n(x)\subset U_x$. Indeed, take any point $z\in \Phi(x)\setminus\Phi_n(x)$ and write it as $z=(1-t)x+ty$ for some $y\in \bar X$ and $0\le t\le f(x,y)$. Find $m\in\w$ with $y\in X_m\setminus X_{m-1}$. Taking into account that $z\notin \Phi_n(x)$ we conclude that $m>n>k$. In this case $t\le f(x,y)\le \frac{\e_{k,m}}m$ and hence $mt\le\e_{k,m}$. The choice of the number $\e_{k,m}$ implies that $(1-mt)x+(mt)y\in X_{k+1}$. Now the choice of the number $n$ yields
$$(1-t)x+ty=(1-\frac1m)x+\frac1m((1-mt)+mty)\in (1-\frac1m)x+\frac1m X_{k+1}\subset U_x.$$
Thus $\Phi(x)\subset U_x\cup\Phi_n(x)$. Now the compactness of $\Phi_n(x)$ guarantees the existence of a finite subfamily $\V\subset\U$ with $\Phi_n(x)\subset\cup\V$. Then $\V\cup\{U_x\}\subset\U$ is a finite subcover of $\Phi(x)$ witnessing the compactness of $\Phi(x)$.

Therefore, the correspondence $\Phi:x\mapsto\Phi(x)$ determines a compact-valued function $\Phi:\bar{X}\to\exp(\bar{X})$. Modifying the preceding argument one can check that this function is continuous with respect to the Vietoris topology on $\exp(\bar X)$.

The condition $[x,y]\subset\Phi(x)\cup\Phi(y)$ for all $x,y\in \bar X$ implies from the equality $f(x,y)+f(y,x)=1$. Finally, the strict convexity of $X$ guarantees that $\Phi(x)\subset  X$ for all $x\in X$. Thus the continuous function $\Phi|X:X\to\exp(X)$ witnesses that the set $X$ is compactly convex.
\end{proof}

\begin{question} Is every upper compactly convex subset of $\IR^n$ compactly convex?
\end{question}

Looking at the preseding statements and known examples of compactly convex sets one can suggest that they are near to being locally compact. The following example shows that it is not so. It relies on the construction of the space $P(K)$ of probability measures on a compact Hausdorff space $K$. We recall that $P(K)=\{\mu\in C^*(K):\mu\ge 0,\;\mu(\bold 1_K)=1\}$ is the subset of the dual Banach space $C^*(K)$, endowed with the $*$-weak topology. Here $C(K)$ is the Banach space of all real-valued continuous functions on $K$ and $\bold 1_K$ is the constant unit function on $K$. By the famous Riesz Representation Theorem elements of $P(K)$ can be viewed as regular Borel probability measures on $K$. Under such an identification the $*$-weak topology on $P(K)$ is generated by the subbase $\{\mu\in P(K):\mu(U)>a\}$ where $a$ is a real number and $U$ runs over open subsets of $K$.

A subset $X$ of a linear topological space $L$ is {\em $\infty$-convex} if for any bounded sequence $(x_n)_{n=1}^\infty\subset X$ and any sequence $(t_n)_{n=1}^\infty$ of positive real numbers with $\sum_{n=1}^\infty t_n=1$ the series $\sum_{n=1}^\infty t_nx_n$ converges to a point from the set $X$, see \cite{BLM}. We recall that a subset $B$ of a linear topological space $L$ is {\em bounded} if for any open neighborhood $U$ of the origin in $L$ there is a number $n\in\IN$ such that $B\subset nU=\{nx:x\in U\}$.
Using the Hahn-Banach Theorem one can show that each convex subset of a finite-dimensional linear topological space is $\infty$-convex.

\begin{example} For a non-empty subset $A$ of a compact Hausdorff space $K$ the subset
$C=\{\mu\in P(K):\exists a\in A$ with $\mu(\{a\})>0\}$ has the following properties:
\begin{enumerate}
\item $C$ is strongly convex, $\infty$-convex and dense in $P(K)$;
\item $C$ is compactly convex;
\item $C$ contains a closed subspace homeomorphic to $A$;
\item each  compact countable subset of $C$ lies in a compact convex subset of $C$;
\item If $A=K$, then $C$ is $\sigma$-compact;
\item If $K$ contains no isolated point, then $C$ is a dense subspace of the first Baire category in $P(K)$.
\end{enumerate}
\end{example}

\begin{proof}
1. The strong convexity, $\infty$-convexity and density of $C$ in $P(K)$ follows immediately from the definition of $C$.

2. Applying Theorem~\ref{strc} we get that $C$ is compactly convex.

3. Identifying each point $x\in K$ with the Dirac measure $\delta_x$ concentrated at $x$, we get an embedding of $K$ into $P(K)$. Then the intersection $K\cap C$ is homeomorphic to $A$.

4. The forth item follows from the $\infty$-convexity of $C$ and the well-known fact stating that the closed convex hull $\con(B)$ of a countable compact subset $B=\{b_n\}_{n=1}^\infty$ in a compact convex subset $A$ of a locally convex space coincides with the set $\con(B)=\{\sum_{n=1}^\infty t_nb_n:t_n\ge 0,\; \sum_{n=1}^\infty t_n=1\}$ of all $\infty$-convex combinations of elements of $B$.

5. Assuming that $A=K$, we shall show that the set  $C$ is $\sigma$-compact. Indeed, in this case
$C=\bigcup_{n=1}^\infty F_n$, where $F_n=\{\mu\in P(K):\exists x\in K$ with $\mu(\{x\})\ge\frac1n\}$. Let us show that each set $F_n$ is closed in $P(K)$.
Take any measure $\mu$ in the closure of $F_n$ in $P(K)$. Let $\mathcal B$ be the family of all open neighborhoods of $\mu$ in $P(K)$. For each $U\in\mathcal B$ find a measure $\mu_U\in F_n\cap U$ and a point $x_U\in K$ such that $\mu_U(\{x_U\})\ge 1/n$. By the compactness of $K\times P(K)$, the set $\{(x_U,\mu_U):U\in \mathcal B\}$ has a cluster point $(x,\mu)$ for some $x\in X$.

We claim that $\mu(\{x\})\ge 1/n$.  Assuming that $\mu(\{x\})<1/n$ by the regularity of $\mu$, we can find a closed neighborhood $W$ of $x$ in $K$ with $\mu(W)< 1/n$. Then $\mu(K\setminus W)>1-1/n$ and $U=\{\eta\in P(K):\eta(K\setminus W)>1-1/n\}$ is an open neighborhood of $\mu$ in $P(K)$. The choice of the cluster point $(x,\mu)$ implies that the neighborhood $W\times U$ of $(x,\mu)$ contains some pair $(x_V,\mu_V)$. Then $\mu_V(W)\ge\mu_V(\{x_V\})\ge 1/n$ and thus $\mu_V\notin U$, which is a contradiction.

This contradiction shows that each $F_n$ is closed in $P(K)$  and consequently, $C=\bigcup_{n=1}^\infty F_n$ is $\sigma$-compact.

6. Now assume that $K$ contains no isolated point. In this case the sets $F_n$ are nowhere dense and consequently, the space $C\subset \bigcup_{n=1}^\infty F_n$ is of the first Baire category.
\end{proof}

\section{Enlarging compactly convex sets}

In this section we shall show that each metrizable compactly convex set in a  linear topological space can be enlarged to a topologically complete compactly convex set. A topological space is called {\em topologically complete} if it is homeomorphic to a complete metric space.

\begin{theorem} Let $X$ be a metrizable compactly convex set in a linear topological space $L$. Then for any $G_\delta$-subset $G\subset L$ containing $X$ there is a topologically complete compactly convex set $C\subset G$ containing $X$.
\end{theorem}

\begin{proof} Let $\Phi:X\to\exp(X)$ be a continuous compact-valued map witnessing the compact convexity of $X$. By Corollaries~\ref{c5}, \ref{c6}, the set $X$ has metrizable locally compact closure $\overline{X}$ in $L$. Let $G_0=G\cap \overline{X}$. The space $G_0$ being a $G_\delta$-subset of the metrizable locally compact space $\overline{X}$ is topologically complete and so is the hyperspace $\exp(G_0)$. By \cite[4.3.20]{En}, the map $\Phi:X\to\exp(X)\subset\exp(G_0)$ admits a continuous extension $\overline{\Phi}:G_1\to\exp(G_0)$ defined on some $G_\delta$-subset $G_1\subset G_0$.

The hyperspace $\exp(G_1)$, being topologically complete, is a $G_\delta$-set in $\exp(G_0)$. Then the continuity of $\overline{\Phi}$ guarantees that the preimage $G_2=\overline{\Phi}^{-1}(\exp(G_1))$ is a $G_\delta$-set in $G_1$ containing $X$. Proceedig by induction, we can define a decreasing sequence $(G_n)_{n=1}^\infty$ of $G_\delta$-subsets of $\overline{X}$ such that $X\subset G_{n+1}=\overline{\Phi}^{-1}(\exp(G_n))$ for all $n$. Then the intersection $C=\bigcap_{n=1}^\infty G_n$ is a $G_\delta$-subset of $G_0\subset G$ that contains $X$ and has the property $\overline{\Phi}(C)\subset\exp(C)$.

Using the continuity of $\overline{\Phi}$ and the density of $X$ in $C$ one can check that $[x,y]\subset\overline{\Phi}(x)\cup\overline{\Phi}(y)\subset C$ for every $x,y\in C$. This means that the set $C$ is compactly convex and the function $\overline{\Phi}|C$ witnesses that.
\end{proof}

This proposition allows as to construct another example of a convex subset of $\IR^3$ that fails to be compactly convex.

\begin{example}\label{ex2} For any disjoint dense subsets  $T_0,T_1$ in the circle $\IT$ the convex subset
$$
X=D\times [0,1]\cup T_0\times\{0\}\cup T_1\times\{1\}$$ of $\IC\times\IR$ fails to be compactly convex because it cannot be enlarged to a convex $G_\delta$-set lying in the $G_\delta$-subset $G=D\times [0,1]\cup \IT\times\{0,1\}$.
\end{example}

Here as before $\IT=\{z\in\IC:|z|=1\}$ is the unit circle and $D=\{z\in\IC:|z|<1\}$ is the open unit disk on the complex plane $\IC$.

\section{Variations of the compact convexity}

In the sequel by $\con(F)$ we denote the convex hull of a subset $F$ of a linear space.
The following proposition shows that the compactly convex space admits a self-enforcement.

\begin{proposition}\label{p7} If $X$ is a compactly convex set, then for every $n\in\IN$ there is a
continuous compact-valued map $\Phi:X\to\exp(X)$ such that $\con(F)\subset \bigcup_{x\in F}\Phi(x)$ for any subset $F\subset X$ of size $|F|\le n$.
\end{proposition}

\begin{proof} This proposition will be proved by induction on $n$. For $n=2$ it follows from the definition of the compact convexity. Assume that for some $n\ge 2$ there is a continuous compact-valued map $\Phi:X\to\exp(X)$ such that $\con(F)\subset \bigcup_{x\in F}\Phi(x)$ for any subset $F\subset X$ of size $|F|\le n$. The map $\Phi$ induces another continuous compact-valued map $$\Psi:X\to\exp(X),\quad \Psi:x\mapsto \Phi(\Phi(x))=\bigcup_{y\in\Phi(x)}\Phi(y).$$ We claim that $\con(F)\subset \bigcup_{x\in F}\Psi(x)$ for any subset $F\subset X$ of size $|F|\le n+1$. Take any point $z\in\con(F)$ and write it as $z=(1-t)x+ty$ for some $t\in[0,1]$, $x\in F$ and $y\in\con(F\setminus\{x\})$. Then
$$\begin{aligned}
z\in&\;[x,y]\subset\Phi(x)\cup\Phi(y)\subset\Psi(x)\cup
\Phi(\con(F\setminus\{x\}))\subset\\
\subset&\;\Psi(x)\cup \Phi(\bigcup_{v\in F\setminus\{x\}}\Phi(v))=\Psi(x)\cup\bigcup_{v\in F\setminus\{x\}}\Phi(\Phi(v))=\bigcup_{v\in F}\Psi(v).
\end{aligned}$$
\end{proof}

In the finite-dimensional case Proposition~\ref{p7} admits further enforcement.

\begin{corollary}\label{p8} If $C$ is a finite-dimensional compactly convex set, then there is a continuous compact valued map $\Phi:C\to\exp(C)$ such that $\con(F)\subset\bigcup_{x\in F}\Phi(x)$ for all $F\subset C$.
\end{corollary}

\begin{proof} This corollary follows from Proposition~\ref{p7} and Carath\'eodory Theorem \cite{Car} according to which for every $F\subset\IR^n$ and $x\in\con(F)$ there is a subset $E\subset F$ of size $|E|\le n+1$ with $x\in\con(E)$.
\end{proof}

Finally, let us pose some questions on the structure of infinite-dimensional compactly convex sets.

\begin{question} Let $C$ be a compactly convex set in a (locally convex) linear topological space.
\begin{enumerate}
\item Is any closed convex subset of $C$ compactly convex?
\item Is there a continuous map $\Phi:C\to\exp(C)$ such that $\con(F)\subset\bigcup_{x\in F}\Phi(x)$ for all (countable) subsets $F\subset C$?
\item Is there a continuous map $\Phi:C\to\exp(C)$ such that all values $\Phi(x)$ are convex and $[x,y]\subset\Phi(x)\cup\Phi(y)$?
\item Is $C$ $\infty$-convex?
\item Does each  compact countable subset $K$ of $C$ lie in a compact convex subset of $C$?
\end{enumerate}
\end{question}

Remark that all these questions have affirmative answers if $C$ is finite-dimensional.

\end{document}